\documentclass[a4paper,12pt,reqno]{amsart}
\usepackage{amssymb,amsfonts,amsmath,amscd,amsthm,latexsym,euscript}
\usepackage{times}

\usepackage{combelowedit} 

\usepackage{tikz}
\usetikzlibrary{shapes.geometric, shapes.misc}

\usepackage[bookmarks=false]{hyperref}
\hypersetup{%
    colorlinks=true,        
    linkcolor=blue,          
    citecolor=blue,         
    urlcolor=blue           
    }

\usepackage[x11names]{xcolor}
\usepackage{listings} 
\definecolor{codegreen}{rgb}{0,0.6,0}
\definecolor{codegray}{rgb}{0.5,0.5,0.5}
\definecolor{codepurple}{rgb}{0.58,0,0.82}
\definecolor{backcolour}{rgb}{0.95,0.95,0.92}
\lstdefinestyle{mystyle}{
    backgroundcolor=\color{backcolour},   
    commentstyle=\color{codegreen},
    keywordstyle=\color{magenta},
    numberstyle=\tiny\color{codegray},
    stringstyle=\color{codepurple},
    basicstyle=\ttfamily\footnotesize,
    breakatwhitespace=false,         
    breaklines=true,                 
    captionpos=b,                    
    keepspaces=true,                 
    numbers=left,                    
    numbersep=5pt,                  
    showspaces=false,                
    showstringspaces=false,
    showtabs=false,                  
    tabsize=2
}
\newtheorem*{theorNo}{Theorem}

\theoremstyle{definition}

\newtheorem*{propNo}{Proposition}

\newtheorem*{corNo}{Corollary}
\newtheorem*{conjectureNo}{Conjecture}

\newtheorem*{resprob}{Research problem}

\newtheorem{ex}{Example}
\theoremstyle{remark}
\newtheorem{rem}{Remark}

\theoremstyle{definition}
\theoremstyle{definition}

\newcommand{\BBR}{\mathbb{R}}\newcommand{\BBC}{\mathbb{C}}
\newcommand{\BBN}{\mathbb{N}} 

\newcommand{\BBZ}{\mathbb{Z}}

\newcommand{\EuA}{{{\EuScript A}}}

\newcommand{\Bone}{{\boldsymbol{1}}}

\newcommand{\bx}{{\boldsymbol{x}}}

\newcommand{\veps}{\varepsilon}

\newcommand{\dd}{\partial}

\DeclareMathOperator{\Van}{Van}

\DeclareMathOperator{\Span}{span}

\DeclareMathOperator{\const}{const}
\DeclareMathOperator{\spanK}{span}

\newcommand{\wedges}{\wedge\cdots\wedge}

\newcommand{\by}[1]{\textrm{{#1}}}
\newcommand{\jour}[1]{\textit{{#1}}}
\newcommand{\vol}[1]{\textbf{{#1}}}
\newcommand{\book}[1]{\textit{{#1}}}

\begin{document}
\title[Iterate Wronskians over $\mathbb{R}^d$: $N$\nobreakdash-bonacci numbers $\gtrsim$ highest total degrees]
{Iterate Wronskians over $\mathbb{R}^d$ as $N$\nobreakdash-ary brackets on $\mathbb{R}[x^1,\ldots,x^d]$:\\[3pt] the $N$\nobreakdash-bonacci numbers bound the highest total degrees}

\author[M.\,G.\,\cb{K}\=eni\cb{n}\v{s}]{Markuss G. \cb{K}\=eni\cb{n}\v{s}${}^{*,\star}$}
\thanks{${}^{*}$\:\textit{Address}:\quad 
Bernoulli Institute for Mathematics, Computer Science \&\ 
Artificial Intelligence, 
University of Groningen, P.O.\,Box\:407, 9700\,AK Groningen, 
The Netherlands.
}
\thanks{${}^{\star}$\:\textit{Address for correspondence} (after 1 September 2026):
ETH Z\"{u}rich -- Department of Mathematics, 
R\"{a}mistrasse 101, CH-8092 Z\"{u}rich, Switzerland.%
}

\author[A.\,V.\,Kiselev]{Arthemy V.\ Kiselev${}^{*,\S}$}
\thanks{${}^{\S}$\:Corresponding author. \textit{E-mail}: \texttt{A.V.Kiselev\symbol{"40}rug.nl}%
}



\subjclass[2010]{
05E15, 
11B39, 
15A15, 
secondary 
05A10, 
05A16, 
17B65
}

\keywords{
Differential polynomial, 
$N$\nobreakdash-ary Lie bracket,
multivariate Wronskian determinant,
Fibonacci numbers,
$N$\nobreakdash-bonacci numbers,
asymptotic growth,
growth of degrees,
Skolem\/--\/Pisot problem.
}

\date{17 August 2026
}

\begin{abstract}
For the algebra $\mathbb{R}[x^1,\ldots,x^d]$ of polynomials in $d\geqslant 1$ variables, regard the complete generalised Wronskian $W_d^k$ of differential order $k\geqslant 1$ over $\mathbb{R}^d$ as the $N=\tbinom{d+k}{d}$\nobreakdash-ary Lie bracket. Take an $N$\nobreakdash-tuple of polynomials, calculate their Wronskian, and keep re-using the newly\/-\/created polynomials to produce more of them. The problem is: how fast do their maximal total degrees grow with the number $n$ of iterations of the bracket\,? Here enter the $N$\nobreakdash-bonacci numbers defined by the recurrence $F^{(N)}_n=F^{(N)}_{n-1}+\cdots+F^{(N)}_{n-N}\in \mathbb{N}$. 
We prove that for any choice of the initial arguments, the sequence of highest total degrees ${\mathsf{D}}^{(N)}_n \geqslant 0$ grows (if at all) asymptotically no faster than the $n$th $N$\nobreakdash-bonacci number: 
$\lim_{n\to+\infty} \bigl({\mathsf{D}}^{(N)}_n / F^{(N)}_n \bigr)<\infty$. We show that for $d=1$ and $k$ odd, the highest polynomial degrees do attain the $N$\nobreakdash-bonacci bound.
\end{abstract}
\maketitle

\subsection*{Introduction}
On the algebra $C^\infty(\BBR^d)$ of smooth functions over $\BBR^d$ with coordinates $\bx=(x^1,\ldots,x^d)$, consider 
the $N$\nobreakdash-linear totally antisymmetric $C^\infty(\BBR^d)$\nobreakdash-valued differential operator $W_d^k=\Bone\wedge \partial_\bx\wedges \partial^k_{\bx\hdots\bx}$ by taking the determinant of $N=\tbinom{d+k}{d}$ functions and their complete set of derivatives up to differential order $k\geqslant 1$ as the rows.\footnote{\label{FootDimJetK}
By convention, the wedge ordering of the derivatives $\dd^{|\sigma|}/\dd\bx^\sigma$, as they mark the determinant rows, is fixed by the lexicographic ordering of the multi-indices~$\sigma$: 
$\varnothing \prec x^1 \prec \ldots \prec x^d \prec x^1x^1 \prec \ldots \prec x^dx^d \prec \ldots \prec \bigl(x^d\bigr)^k$.
At every differential order~$\ell$ given in the range $0\leqslant \ell\leqslant k$,
there are $\binom{\ell + (d-1) }{ d-1 }$ such monomials~$\tau$ of degree~$\ell$ that do not coincide pairwise, hence as many essentially different derivatives $\dd^{|\tau|}/\dd\bx^\tau$ of order~$\ell$. (Note that on $C^\infty(\BBR^d)$, mixed derivatives coincide.) Over all the orders~$\ell$ up to and reaching~$k$, the sum of binomial coefficients yields the size of the Wronskian determinant,
$N=\sum_{\ell=0}^k \binom{\ell + d-1 }{ d-1 } = \binom{k+d}{d}$.

The formula $N=\sum_{\ell=0}^k \binom{\ell + d-1 }{ d-1 } = \binom{k+d}{d}$ expressing the size of the Wronskian determinant is proven by induction over~$\ell$. The base of induction is $\binom{0 + d-1 }{ d-1 } = 1 = \binom{d}{d}$. The inductive step uses the binomial identity $\binom{m}{d} + \binom{m}{ d-1 } = \binom{m+1}{d}$; here $m$ and $d$ are positive integers. Running the differential order $\ell$~up, after $k$~steps we accumulate $\binom{k+d}{d}$ partial derivatives (in $d$ variables) of all nonnegative orders not exceeding~$k$. This proves the claim.} 
The $N$-ary operator $W_d^k$ is the \emph{complete generalised Wronskian determinant} of differential order~$k$ over~$\BBR^d$.

For instance, the only 
Wronskian of two arguments is the operator $W^{0,1} \mathrel{{:}{=}} W_{k=1}^{d=1}=\Bone\wedge \partial/\partial x$, which, evaluated at 
two functions $f,g\in C^\infty(\BBR)$,\label{pWronskD1K1} %
is $W^{0,1}(f,g)=\det\bigl(\begin{smallmatrix} f&g\\f'&g' \end{smallmatrix}\bigr)$.
Yet 
there are two ternary ($N=3$) operators, namely the usual Wronskian
$W^{0,1,2} \mathrel{{:}{=}} W_{d=1}^{k=2}=\Bone\wedge\partial_x\wedge\partial^2_{x}$ 
and the complete generalised Wronskian $W_{d=2}^{k=1}=\Bone\wedge \partial_x\wedge \partial_y$; these structures are given by the determinants
\begin{align}\label{EqExTernaryD1D2}
    W^{0,1,2}(f,g,h)(x)&=\begin{vmatrix}
        f&g&h\\ f'&g'&h'\\ f''&g''&h''
    \end{vmatrix} &&
    \text{and} &
    W_{d=2}^{k=1}(f,g,h)(x,y)&=\begin{vmatrix}
        f&g&h\\ f_x&g_x&h_x \\ f_y&g_y&h_y
    \end{vmatrix},
\end{align}
where $f,g,h$ are suitable functions in $d$~variables, and 
$f_x \mathrel{{:}{=}}\partial f/\partial x$ and $f_y \mathrel{{:}{=}}\partial f/\partial y 
$. 

Let us keep in mind that the Wronskian determinant 
$W^{0,1,\ldots,N-1}\bigl(x^{\alpha_1},\ldots,x^{\alpha_N}\bigr)$
of monomials in a single variable~$x$ again is a monomial in~$x$, and its coefficient is the Vandermonde determinant (relative to the Wronskian and all the lexicographic orderings in~it),
\[
\det\Van_{d=1}^{k=N-1}\bigl(\alpha_1,\ldots,\alpha_N\bigr) =
\det\Van^{0,1,\ldots,N-1}\bigl(\alpha_1,\ldots,\alpha_N\bigr) =
\prod_{1\leqslant i < j \leqslant N} \bigl( \alpha_j - \alpha_i \bigr).
\]
Likewise, the 
generalised Wronskian determinant of monomials in $d$ variables is again a monomial.
In turn, the complete generalised Vandermonde determinants are defined 
(see~\cite{FinDim}) 
by taking the complete generalised Wronskian determinants $W_{d\geqslant1}^{k\geqslant1}$ of $N=\binom{d+k}{d}$ monomials in $d$~variables and then, 
by inspecting the determinant coefficient preceding the monomial.

\begin{ex}\label{ExWronskVanderDimD}
For the two Wronskian determinants in Eq.~\eqref{EqExTernaryD1D2}
, we have that
$ W^{0,1,2}(x^a,x^b,x^c) 
 = \left|\begin{smallmatrix}
         1&1&1 \\ a&b&c \\ a^2 & b^2 & c^2 
     \end{smallmatrix}\right| \cdot x^{a+b+c-3}
$
and
$ W_{d=2}^{k=1}(x^{\ell_1}y^{\ell_2},x^{m_1}y^{m_2},x^{n_1}y^{n_2}) 
=
     \left|\begin{smallmatrix}
         1&1&1 \\ \ell_1 & m_1& n_1 \\ \ell_2 & m_2 & n_2        
     \end{smallmatrix}\right| \cdot x^{\ell_1+m_1+n_1 - 1} \, y^{\ell_2+m_2+n_2 - 1}$.
The rows of (generalised) Vandermonde's determinants are marked by the monomial multi\/-\/indices $1,x,x^2$ and $1,x,y$ that mark the rows of (generalised) 
Wronskians $W^{0,1,2}=\mathbf{1}\wedge\dd_x\wedge\dd_x^2$ and $W_{d=2}^{k=1}=\mathbf{1}\wedge\dd_x\wedge\dd_y$, respectively.
\end{ex}

Appearing in this note in the context of (generalised) Wronskians, the Fibonacci numbers --\,and their generalisations, the $N$-bonacci numbers\,-- serve two tasks:
(A) these numbers' asymptotic growth rate yields an upper bound for the degrees of polynomials which \emph{may occur} in the constructions under study;
(B) showing up in several examples (parametrised by the dimension $d\geqslant 1$ and differential order $k\geqslant 1$), the Fibonacci or $N$-bonacci numbers encode the degrees of polynomials which the Wronskian determinants \emph{do produce} from a suitably taken initial set of monomials (by following an explicitly described scheme for iterating the $N$-ary bracket). In other words, the asymptotic upper bound is attained in many cases.

   Let us summarise these results, respectively:

\smallskip
\textbf{(A):\quad The polynomial degrees cannot grow arbitrarily fast.}\quad
Let us \emph{iterate} the generalised Wronskian determinant $W_{d\geqslant 1}^{k\geqslant 1}$ of differential order $k$ over $\mathbb{R}^d$ as the $N=\tbinom{d+k}{d}$\nobreakdash-ary bracket on the space of polynomials $\mathbb{R}[x^1,\ldots,x^d]$ in $d$ variables. Fix some $N$\nobreakdash-tuple of 
monomial arguments to start.\footnote{\label{FootSeqArg}
By convention, the default ordering of monomial arguments is lexicographic: the $d$ variables obey $x^1\prec\ldots\prec x^d$ and polynomial degrees occur in nondecreasing order: see Eqs.~(\ref{eq:FiboBase}--\ref{EqNbonMonomNbonCoeff}) in Examples~\ref{ex:FiboBase}--\ref{ex:Nbon} on pages~\pageref{ex:FiboBase}--\pageref{EqNbonMonomNbonCoeff}.} 
Every calculation of the monomials' Wronskian produces the monomial (possibly, previously not available) which can from now on be (re)\/used as one of the $N$ arguments of the bracket.
There can be many schemes to re\/-\/use the new monomial, provided its coefficient did not vanish: e.g., the old leftmost argument of the Wronkskian is dropped and the new monomial, now taken with unit coefficient, becomes the last argument. By following all the possible schemes to use the so far attained monomials in the ordered $N$-tuples of arguments, we inspect the top (i.e.\ highest appearing) degree of resulting monomials: 
after $n\in\mathbb{N}$ iterations, how fast does it grow\,?
We prove that this $n$th top 
total degree grows (if at all\footnote{\label{FootIfGrows}
The polynomial degrees can remain bounded uniformly over all $n\in\BBN$.
For example, let $d=1$ and $k=1$, so $N=2$ (see the material 
on p.~\pageref{pWronskD1K1} before 
Eq.~\eqref{EqExTernaryD1D2}), and start the iterations by taking the monomials $a_0=1$ and $a_1=x$: we have that $W_{d=1}^{k=1}\bigl(1,x\bigr) = 1 = a_1 \in \spanK_{\BBR}\langle a_0,a_1\rangle$, thus producing now new degrees in the variable~$x$.\quad
Another example over $\Bbbk=\BBR$ or~$\BBC$ (this time, the Lie algebra $\mathfrak{sl}(2,\Bbbk)$ is perfect, $\spanK_{\Bbbk}[\EuA,\EuA] = \EuA$, with the Lie bracket $[{\cdot},{\cdot}] = \mathbf{1}\wedge \dd_x$): start with the monomials $e\mathrel{{:}{=}} 1$ and $f\mathrel{{:}{=}} -x^2$, whence put $h\mathrel{{:}{=}} [e,f]=-2x$; the triple of quadratic polynomials $e,h,f\in\Bbbk_2[x]$ satisfies the Chevalley relations 
$[h,e]=2e$ and $[h,f]=-2f$.\quad
For larger $N > 2$ over the lowest base dimension $d=1$, the previous example of three\/-\/dimensional Lie algebra $\mathfrak{sl}(2,\Bbbk)$ was generalised in~\cite{ForKac} to the countable sequence of $(N+1)$-\/dimensional strongly homotopy Lie algebras $\bigl(\Bbbk_N[x]$,$W_{d=1}^{k=N}\bigr)$, here $\Bbbk=\BBR$ or~$\BBC$ again.
In our recent work~\cite{FinDim} we describe --\,under the natural assumption $\Bbbk_k[x^1,\ldots,x^d]\subseteq\EuA$\,-- all such finite\/-\/dimensional subalgebras $\EuA \subsetneq \Bbbk[x^1,\ldots,x^d]$, closed w.r.t.\ the Wronskian brackets $W_{d \geqslant 1}^{k \geqslant 1}$ on the spaces of polynomials in $d$~variables.
In all of these examples, during the iteration of the $N$-ary bracket by following any schemes for re\/-\/using its output, the polynomial degrees cannot exceed the maximal degree of monomials in the algebra~$\EuA$ at hand.}\,)
asymptotically not faster than the $n$th $N$\nobreakdash-bonacci number. (Here the Fibonacci sequence has window width 2 in the recurrence $F_n=F_{n-1}+F_{n-2}$, while tribonacci numbers have window width 3; the $N$\nobreakdash-bonacci numbers satisfy $F^{(N)}_n = F^{(N)}_{n-1}+\cdots + F^{(N)}_{n-N}$.) 
The Fibonacci numbers $F_n$ grow asymptotically as $F_n\sim \varphi^n$, with $\varphi=(1+\sqrt{5})/2\approx 1.618\ldots$ the golden ratio; the $N$\nobreakdash-bonacci numbers $F^{(N)}_n$ grow faster:~$\alpha_N\in [\varphi,2)$~and~\mbox{$F^{(N)}_n\sim \alpha_N^n$}.

\smallskip
\textbf{(B):\quad The polynomial degrees' upper bound can be attained.}\quad
We 
investigate whether over arbitrary base dimension $d\geqslant 1$ and 
for the differential order $k\geqslant 1$ of the Wronskian bracket,
the upper bound can (or cannot) be attained.
In the set-up $d\geqslant 1$ and in the lowest order $k=1$ 
in Examples~\ref{ex:FiboBase} and~\ref{ex:FiboOverDim} on p.~\pageref{ex:FiboBase} onwards, we produce a sequence of monomials, with non\/-\/zero coefficients, whose degrees in the base variables are the Fibonacci numbers (up to constants).
For $d=1$ and any odd order $k$ (so $N=k+1$ even), in Example~\ref{ex:Nbon} on p.~\pageref{ex:Nbon} we find $N$ monomial arguments and a way to re-use the newly produced monomials such that their (integer) coefficient never vanishes and the upper bound is attained: the monomial degrees equal the $N$\nobreakdash-bonacci numbers plus a constant.\footnote{\label{FootDegShiftDimD}
The (Laurent) monomial degree shifts --\,in the context of Wronskian determinants $W_{d\geqslant 1}^{k\geqslant 1}$ as the $N$-ary brackets which generalise the Witt algebra structure on $\Bbbk[[x]]$\,-- are discussed in more detail in~\cite{ForKac} (over the base dimension $d=1$) and in~\cite{FinDim} (over higher dimensions $d>1$).} 
(The Fibonacci case is $d=1=k$ and the degree shift is $+1$.) 
We conjecture that for all $N\geqslant 2$ the higher bound~$\alpha_N^n$ can asymptotically be attained as $n\to +\infty$.

But what if, for a chosen way to re\/-\/use the new monomials in the Wronskian bracket, the Vandermonde determinant standing as the \emph{coefficient} of the output hits zero\,? This depends on the initially taken $N$-tuple of arguments and on the scheme -- of re\/-\/using the new monomials -- that determines the linear recurrence relation for the coefficients.
We experimentally find many monomial sequences (over dimensions $d=1$,\ $2$,\ $3$,\ $4$ and differential orders $k=1$,\ $2$,\ $3$) such that their coefficients do not vanish as far as $n=3\,000 
$ iterations 
(see Example~\ref{ex:computerAlgebra} on p.~\pageref{ex:computerAlgebra}). 
To determine whether the coefficients \emph{never} vanish for all $n\in\mathbb{N}$ is a particular case of the Skolem\/--\/Pisot problem, which is NP\nobreakdash-hard.

\medskip
This note is structured as follows: in Examples~\mbox{\ref{ex:FiboBase}--\ref{ex:Nbon}} we showcase the degree growth as Fibonacci or $N$\nobreakdash-bonacci numbers.
These examples introduce 
the research problem; 
our main theorem controls the rapidity of growth for the monomial degrees, relating it to the Fibonacci (or $N$-bonacci) numbers.
An immediate corollary 
--\,in the context of Kirillov's problem 
from~\cite{KirillovKontsevich}\,-- bounds the growth of \emph{dimension} for the subspaces generated in the algebra of polynomials by the $n$-fold iteration of the Wronskian as the $N$-ary bracket.
(For reader's convenience, 
in the proposition on p.~\pageref{prop:NbonGrowth} we recall
--\,and in Appendix~\ref{SecAppReProof} we re\/-\/prove\,-- 
the rapidity of asymptotic growth for the $N$-bonacci numbers;
our bound is asymptotically $\exp(1)/2\approx 1.359$ times less sharp than Wolfram's in~\cite{Wolfram98} yet our proof is more geometric.)
We conclude this note with a conjecture, itself based on the many examples:
there exists an initially taken $N$-tuple of monomials in $d$~variables such that by iterating the Wronskian bracket and re\/-\/using its output, one obtains the sequence of monomials whose degrees grow as fast as the $N$-bonacci numbers
and whose coefficients, given by the generalised Vandermonde determinants, never vanish.

\medskip
We place this note in a broader context of
  (\textit{i})  $N$-ary Lie structures, in particular strongly homotopy 
Lie algebras of (non)\/polynomial functions on~$\BBR^d$ with the (in)complete generalised Wronskian determinants as the $N$-ary brackets (see~\cite{FinDim,ForKac,PRG25,Yer25} and references therein), and
  (\textit{ii})  Kirillov's problem about Lie algebras of intermediate growth, specifically about the rapidity of growth for the dimensions of vector spaces which are spanned by the output of $n\in\BBN$ iterations of the Lie bracket, starting with a generic tuple of smooth vector fields\footnote{\label{FootBinaryNary}
The upgrade of the binary Lie bracket for vector fields (their differential order $p=1$ is minimal) to $N$-ary Lie\/-\/type brackets is done by using higher differential order operators on the line~$\BBR$ and their generalisations over~$\BBR^{d\geqslant 1}$ (see~\cite{ForKac,PRG25} and~\cite{cpWron}).} 
(see~\cite{KirillovKontsevich} as well as~\cite{FinDim} and references therein).

\centerline{\rule{3.5in}{0.7pt}}

Let us have three examples: over the various base dimensions $d\geqslant 1$, 
these examples suggest which polynomial degrees can actually be produced for monomials in $d$~variables by iterating the (generalised) Wronskian determinant of a chosen differential order~$k\geqslant 1$.

\begin{ex}[$d=1=k$: Fibonacci]\label{ex:FiboBase}
    Put $W^{0,1}=\Bone \wedge \partial/\partial x$ as the Wronskian of first differential order $k=1$ over $\BBR^1\ni x$. Denote by $F_n$ the $n$th Fibonacci number: $F_0=0$, $F_1=1$, and $F_n=F_{n-1}+F_{n-2}$. Then the Wronskian of monomials $x^{F_n+1}$ and $x^{F_{n+1}+1}$, whose degrees are successive Fibonacci numbers shifted by $+1$, is
    \begin{equation}\label{eq:FiboBase}
        W^{0,1} \bigl( x^{F_n+1}, x^{F_{n+1}+1} \bigr) = F_{n-1}\, x^{F_{n+2}+1}.
    \end{equation}
In other words, 
the result is a monomial with the next Fibonacci number plus shift $+1$ as the degree;
the new monomial's 
coefficient is not zero for~$n>1$.

Notice that, by taking the newly\/-\/created monomial as the last subsequent argument and dropping the lowest\/-\/degree old first argument, i.e.\ calculating $W^{0,1}\bigl( x^{F_{n+1}+1}, x^{F_{n+2}+1} \bigr)$, one constructs the sequence of monomials --\,starting from the powers $x^2=x^{F_2+1}$ and $x^3=x^{F_3+1}$\,-- whose degrees are successive Fibonacci numbers. 
\end{ex}

\begin{ex}[$d>1$, $k=1$: Fibonacci]\label{ex:FiboOverDim}
Let us generalise~\eqref{eq:FiboBase}, 
now over $\BBR^d\ni (x^1,\ldots,x^d)=\bx$. Put $W_d^{k=1}=\Bone \wedge \partial/\partial x^1\wedges \partial/\partial x^d$ as 
the first\/-\/differential\/-\/order 
Wronskian over~$\BBR^d$. 
Let us choose the $N=d+1$ monomial arguments to iterate 
the $N$-ary bracket~$W_{d>1}^{k=1}$.

\smallskip\noindent$\bullet$\quad
If $d=2$ (so we put $x\mathrel{{:}{=}} x^1$ and $y\mathrel{{:}{=}} x^2$ for brevity, cf.\ Example~\ref{ExWronskVanderDimD}), then $N=3$, and for any $n\geqslant 5$ we take the three arguments
$xy$, $x^{F_n} y^{F_{n-1}}$, and $x^{F_{n+1}} y^{F_{n}}$.
Their Wronskian determinant equals
\[
\begin{vmatrix}
xy & x^{F_n} y^{F_{n-1}} & x^{F_{n+1}} y^{F_{n}} \\
y & F_n\, x^{F_n - 1} y^{F_{n-1}} & F_{n+1}\,x^{F_{n+1} - 1} y^{F_{n}} \\
x & F_{n-1}\,x^{F_n} y^{F_{n-1} - 1} & F_{n}\,x^{F_{n+1}} y^{F_{n} - 1}
\end{vmatrix}
= \begin{vmatrix}
1 & 1 & 1 \\ 1 & F_n & F_{n+1} \\ 1 & F_{n-1} & F_n 
\end{vmatrix} \cdot x^{F_n + F_{n+1} } \cdot y^{F_{n-1} + F_{n} }.
\]
The polynomial degrees in the new monomial $x^{F_{n+2}} y^{F_{n+1}}$ are given by the Fibonacci numbers with the same pattern as before: 
from the arguments' degree pairs $(F_n,F_{n-1})$ and $(F_{n+1},F_{n})$ 
we obtain the degree pair $(F_{n+2},F_{n+1})$ for $x$ and~$y$.
Note that the degrees of either $x$ or~$y$ are both strictly increasing and never vanish as 
$n$~tends to~$+\infty$. 
The $(3\times 3)$-size determinant coefficient of the newly produced monomial expands to\footnote{\label{FootCassini}
Transforming the first term amounts to Cassini's identity:
\[
\left|\begin{smallmatrix} F_n & F_{n+1} \\ F_{n-1} & F_n \end{smallmatrix}\right|
= \left|\begin{smallmatrix}
F_{n-1} + F_{n-2} & F_n + F_{n-1} \\ F_{n-1} & F_{n-1} + F_{n-2}
\end{smallmatrix}\right|
= \left|\begin{smallmatrix} 1 & 1 \\ 1 & 0 \end{smallmatrix}\right| 
\cdot \left|\begin{smallmatrix} F_{n-1} & F_n \\ F_{n-2} & F_{n-1} \end{smallmatrix}\right| 
= \left|\begin{smallmatrix} 1 & 1 \\ 1 & 0 \end{smallmatrix}\right| ^q
\cdot \left|\begin{smallmatrix} F_{n-q} & F_{n-q+1} \\ F_{n-q-1} & F_{n-q} \end{smallmatrix}\right|.
\]
Now let $q\mathrel{{:}{=}} n-1$, thus reaching
\[
(-1)^{n-1}\, \left|\begin{smallmatrix} F_1 & F_2 \\ F_0 & F_1 \end{smallmatrix}\right|
= (-1)^{n-1}\, \left|\begin{smallmatrix} 1 & 1 \\ 0 & 1 \end{smallmatrix}\right|
= (-1)^{n-1}.
\]
In the expansion's second term, we use the Fibonacci recurrence relation
$F_{n+1} = F_{n} + F_{n-1}$ (and likewise for the indices $n$ and $n-1$) three times:
$F_{n+1} -2F_{n} + F_{n-1} = F_{n-3}$.}
\[
\left|\begin{smallmatrix} F_n & F_{n+1} \\ F_{n-1} & F_n \end{smallmatrix}\right|
 + \bigl( F_{n+1} - 2F_{n} + F_{n-1} \bigr)
= F_{n-3} - (-1)^n \neq 0 \qquad (n\geqslant 5).
\]
We conclude that
\begin{equation}\label{EqD2CoeffFiboPlus}
W_{d=2}^{k=1}\bigl( xy, x^{F_n} y^{F_{n-1}}, x^{F_{n+1}} y^{F_{n}} \bigr)
= x^{F_{n+2}} y^{F_{n+1}} \cdot \bigl( F_{n-3} - (-1)^n \bigr);
\end{equation}
the coefficient does not vanish\footnote{\label{FootExD2Low}
The coefficient equals $+1$ at $n=3$ but it vanishes incidentally at $n=4$.
If $n=1$ or $n=2$, either Wronskian $W_{d=2}^{k=1}\bigl(xy,x^1 y^0,x^1 y^1\bigr)$
and $W_{d=2}^{k=1}\bigl(xy,x^1 y^1,x^2 y^1\bigr)$ has two repeated arguments, hence it vanishes identically.
At $n=0$ with $F_{-1}=1$, the Wronskian is $W_{d=2}^{k=1}\bigl(xy,x^0 y^1,x^1 y^0\bigr) = xy$.} 
for all $n\geqslant 5$.
The index of Fibonacci number $F_{n-3}$, still showing up in the monomial coefficient, retards w.r.t.\ the index in Eq.~\eqref{eq:FiboBase} over dimension $d=1$ (now we have $d=2$); in the future formula~\eqref{eq:FiboOverDim} over $d\geqslant 3$ the now\/-\/retarding Fibonacci term will be lost yet the currently emerging Cassini term will persist.

\smallskip\noindent$\bullet$\quad
If $d\geqslant 3$, we let the first $d-2$ arguments 
be the unit~$1$ followed by the $d-3$ monomials (if any, appearing for $d\geqslant 4$) $x^1$,\ $\ldots$,\ $x^{d-3}$ 
of total degree~$\leqslant 1$ in the variables $x^i$; 
next is the degree\/-\/three product of the last three coordinates $x^{d-2}x^{d-1}x^d$, and the last two arguments consist of base variables raised to Fibonacci numbers.
(For instance, when the base dimension is $d=3$ with $(x,y,z)\in\BBR^3$, take $x^{F_{n}} y^{F_{n-1}} z^{F_{n-2}}=\prod_{i=1}^d (x^i)^{F_{n-i+1}}$ and $\prod_{i=1}^d (x^i)^{F_{n-i+2}}$ as the last two arguments.)
Then we claim that
  \begin{equation}\label{eq:FiboOverDim}
        W^{k=1}_d \Bigl(1,x^1,\ldots,x^{d-3}, x^{d-2}x^{d-1}x^d, \prod_{i=1}^d (x^i)^{F_{n-i+1}}, \prod_{i=1}^d (x^i)^{F_{n-i+2}} \Bigr) = (-1)^{n-d} \prod_{i=1}^d (x^i)^{F_{n-i+3}},
  \end{equation}
i.e.\ the result is a monomial whose base variables are raised to the \emph{next} Fibonacci numbers in each variable (compare the index shifts $+1,+2\to +3$), and the coefficient does not vanish.

Here we produce the sequence of monomials such that the degree in each base coordinate diverges to infinity exponentially fast (as the $n$th Fibonacci number, with $n\to+\infty$).
\end{ex}

\begin{proof}[Proof of equality~\textup{\eqref{eq:FiboOverDim}}]
  By \cite[Th.
\:7 on p.\,12]{FinDim} about the Wronskian of monomial arguments, the value 
of the Wronskian in~\eqref{eq:FiboOverDim} is the monomial whose degree in coordinate $x^i$ is $(1+F_{n-i+1}+F_{n-i+2})-1=F_{n-i+3}$ and whose coefficient is the (generalised Vandermonde) determinant
    \begin{equation*}
        \det\left|\begin{array}{c|ccc|ccc}
        1 & 1 & \cdots  &  1 & 1 & 1     & 1 \\ \hline
        0 & 1 & \cdots  &  0 & 0 & F_{n} & F_{n+1} \\ 
    \vdots& \vdots & \ddots  &\vdots &\vdots  &\vdots  &\vdots  \\ 
        0 & 0 & \cdots  &  1 & 0 & F_{n-d+4} & F_{n-d+5} \\ \hline
        0 & 0 & \cdots  &  0 & 1 & F_{n-d+3} & F_{n-d+4} \\ 
        0 & 0 & \cdots  &  0 & 1 & F_{n-d+2} & F_{n-d+3} \\ 
        0 & 0 & \cdots  &  0 & 1 & F_{n-d+1} & F_{n-d+2} 
        \end{array}\right|
        = 
        \begin{vmatrix}
            1 & F_{n-d+3} & F_{n-d+4} \\ 
            1 & F_{n-d+2} & F_{n-d+3} \\ 
            1 & F_{n-d+1} & F_{n-d+2}
        \end{vmatrix},
    \end{equation*}
which simplifies to Cassini's identity, $F_{n-d+1}^2-F_{n-d}\,F_{n-d+2}=(-1)^{n-d}$,
    as required.
\end{proof}

\begin{ex}[$d=1$, $k\geqslant 1$: $N$\nobreakdash-bonacci]\label{ex:Nbon}
    Let us return to base $\BBR^1\ni x$ and see if faster --\,than 
Fibonacci's\,-- growth of the monomial degrees is possible. Take an odd differential order $k\in \{1,3,5,\ldots\}$, so that the Wronskian of even arity $N=k+1$ becomes $W^{0,1,\ldots,N-1}=\Bone \wedge \partial_x\wedges \partial^{N-1}_x$. The $N$\nobreakdash-bonacci numbers $F_n^{(N)}$ are given by $F^{(N)}_0=\cdots=F^{(N)}_{N-2}=0$, $F^{(N)}_{N-1}=1$, and $F^{(N)}_n=F^{(N)}_{n-1}+\cdots + F^{(N)}_{n-N}$, i.e.\ summation is taken over the preceding $N$ numbers in the sequence. 
Then the Wronskian of $N$ monomials --\,whose degrees are $N$ successive $N$\nobreakdash-bonacci numbers plus the shift $N/2$\,-- is the monomial with the next $N$\nobreakdash-bonacci number in the degree,
\begin{equation}\label{EqNbonMonomNbonCoeff}
        W^{0,1,\ldots,N-1} \bigl( x^{F^{(N)}_{n}+N/2}, 
        \ldots,x^{F^{(N)}_{n+N}+N/2} \bigr) = 
        \prod_{0\leqslant i<j<N} \Bigl( \underbrace{F^{(N)}_{n+j} - F^{(N)}_{n+i} }_{>0} \Bigr) \cdot 
        x^{F^{(N)}_{n+N+1} + N/2};
\end{equation}
the monomial's coefficient does not vanish\footnote{When $n=N-1$, there is a repeated argument $x^{1+N/2}$, hence the Wronskian determinant vanishes.}
for $n\geqslant N$.    

    We thus produce a sequence of monomials whose degrees grow faster than the Fibonacci numbers; their growth is again exponential, but asymptotically slower than $2^n$; see the Proposition on p.~\pageref{prop:NbonGrowth} about their growth. (The case $2^n$ corresponds to the $\infty$-bonacci numbers, where summation is taken over \emph{all} previous terms.)
\end{ex}

\begin{resprob}\label{ResProb}
    Take the complete generalised Wronskian $W_d^k = \Bone\wedge \partial_\bx \wedges \partial^k_{\bx\hdots\bx}$ of differential order $k\geqslant 1$ 
over $\BBR^{d\geqslant 1} \ni \boldsymbol{x}$ as the $N$\nobreakdash-ary Lie bracket, $N=\tbinom{d+k}{d}$. Pick some set of initially taken polynomial arguments from $\BBR[x^1,\ldots,x^d]$, denoting their span by $\EuA_0$, and put ${\mathsf{D}}_0$ to be the highest total polynomial degree among them. Plug the polynomials from $\EuA_0$ into the Wronskian, obtaining (possibly new) polynomials, the span of which we denote by $\EuScript{B}_1$. Put $\EuA_1 \mathrel{{:}{=}} \Span(\EuA_0\cup \EuScript{B}_1)$. 
Now denote by ${\mathsf{D}}_1$ the highest total degree among the polynomials from $\EuA_1$. Repeat this process of iterating the $N$\nobreakdash-ary bracket using polynomials from $\EuA_1$, obtaining the (sub-)spaces of polynomials $\EuA_2$,\ $\EuA_3$,\ $\ldots$ and the highest total degrees ${\mathsf{D}}_2$,\ ${\mathsf{D}}_3$,\ $\ldots$ of polynomials among them; let us introduce the notation ${\mathsf{D}}_n={\mathsf{D}}_n^{(N)}$ to emphasise the degrees' dependence on the arity~$N$.

    The problem is to determine how fast the sequence $\{{\mathsf{D}}^{(N)}_n\}_{n\in\BBN}$ grows (if at all) as $n\to+\infty$. In particular, \emph{what is the fastest possible growth of $\{{\mathsf{D}}^{(N)}_n\}_{n\in\BBN}$ among all choices of the initial arguments,
$\EuA_0\subsetneq \BBR[x^1,\ldots,x^d]$\,}?
\end{resprob}

Our main result is 
an upper bound for the asymptotic growth of the (poly- or) monomials' degrees $\{{\mathsf{D}}^{(N)}_n\}_{n\in\BBN}$. 

\begin{theorNo}\label{thm:main}
    Let the base dimension $d=\dim \BBR^d\geqslant 1$ and differential order $k\geqslant 1$ of the complete generalised Wronskian $W_d^k$ be arbitrary; put $N=\tbinom{d+k}{d}$. For any choice of initial arguments for the Wronskian, the sequence of highest total degrees $\{{\mathsf{D}}^{(N)}_n\} 
    $ (and the highest degree in each variable) \emph{g\underline{rows }(\underline{if at all})\underline{ as}y\underline{m}p\underline{toticall}y\underline{ no faster than the $N$\nobreakdash-bonacci numbers}.} 
\end{theorNo}

(The proof is given in Appendix~\ref{SecAppProofThm}.)

Now we can 
bound the growth in dimension of the subspaces $\EuA_0$,\ $\EuA_1$,\ $\ldots$; the growth in dimension, $\dim(\EuA_n)$, is Kirillov's problem~\cite{KirillovKontsevich} extended to the case of (polynomial) $N$\nobreakdash-ary Lie algebras (see~\cite{FinDim} for the answer in the base case --- for certain 
algebras that do not grow at all).

\begin{corNo}[towards Kirillov's problem]\label{corr:dimgrowth}
    For any dimension $d\geqslant 1$ and differential order $k\geqslant 1$, so that $N=\tbinom{d+k}{d}$, and any initially taken polynomial arguments for the Wronskian, the dimensions $\dim (\EuA_n)$ of the subspaces of polynomials created after $n>0$ iterations of the $N$\nobreakdash-ary bracket grow (if at all) no faster than the following asymptotic bound, as $n\to+\infty$:
\begin{equation}\label{EqDimBoundAsympt}
\log\dim(\EuA_n) \lesssim n\cdot\log (\alpha_N^d) + O(1) < n \cdot \log (2^d) + O(1),
\end{equation}
where $\alpha_N =\displaystyle\lim_{n\to+\infty} (F^{(N)}_{n+1} / F^{(N)}_n ) \in [\varphi,2)$ describes the asymptotic growth of the $N$\nobreakdash-bonacci numbers,
see the proposition on the next page. 
(Here $\varphi=(1+\sqrt{5})/2\approx 1.618\ldots$ is the golden ratio.)
\end{corNo}

\begin{proof}[Proof (of the Corollary).]
  The dimension of the space of polynomials in $d$ variables up to degree $m$ is given
(see footnote~\ref{FootDimJetK} on p.\,\pageref{FootDimJetK})  
by the binomial coefficient 
$\tbinom{d+m}{d} = m^d / d! + O\bigl(m^{d-1}\bigr)$. 
As the degrees of all monomials created after $n\in\BBN$ iterations of the Wronskian $N$\nobreakdash-ary bracket do not exceed ${\mathsf{D}}^{(N)}_n$, 
it follows that 
\[
\dim (\EuA_n)\leqslant \tbinom{d+{\mathsf{D}}^{(N)}_n}{d} 
 = \bigl( {\mathsf{D}}^{(N)}_n \bigr)^d \bigr/ d! 
  + O \Bigl( \bigl( {\mathsf{D}}^{(N)}_n \bigr)^{d-1} \Bigr).
\]
As the highest degrees ${\mathsf{D}}^{(N)}_n$ grow asymptotically no faster than the $N$\nobreakdash-bonacci numbers,  
$\log {\mathsf{D}}^{(N)}_n \lesssim \log\alpha_N^n + O(1)$ as $n\to +\infty$,
we conclude 
that 
\[
\log \dim(\EuA_n) \lesssim d\cdot\log \bigl({\mathsf{D}}^{(N)}_n\bigr) + O(1)
\lesssim n\cdot \log (\alpha_N^d) + O(1) < n\cdot \log (2^d) + O(1), 
\]
as claimed.\footnote{\label{FootNeedExpData}
At this point, the next step along this line of research would be a sufficient number of (computationally costly) examples revealing the actual rapidity of growth for the dimensions $\dim \bigl( \EuA_{n\geqslant 1} \bigr)$, starting from generic $(M\geqslant N)$-\/tuples of initially taken polynomials $a_i\in\EuA_0$.
In particular, the upper bound 
$\dim \bigl( \EuA_{n} \bigr) \leqslant 
  \const(d) \cdot \exp\bigl( \const(d)\cdot n^1 \bigr)$, de facto available from~\eqref{EqDimBoundAsympt}, might far exceed the actual asymptotic growth rates
--- due to the integer power $n^1$ under the exponent (in contrast with the fractional power $n^{2/3}$ in the seminal work~\cite{KirillovKontsevich} over the base dimension $d=1$). From the proof of asymptotic bound~\eqref{EqDimBoundAsympt} we now see that the dimension $d\geqslant 1$ shows up explicitly in either constant around the exponential function of the upper bound.}
\end{proof}

To provide more details about the limit values $\alpha_N$,
let us recall some basic facts about linear integer recurrences with constant coefficients. 
Define the integer sequence \mbox{$y_n\in \BBZ$} recursively by $y_n=a_1 y_{n-1} + \cdots a_N y_{n-N} + b$, where $a_1,\ldots,a_N,b\in \BBZ$ are constants, and the sequence begins from some initial datum $y^0_1,\ldots,y^0_N$; the sequence is homogenised, $x_n=a_1 x_{n-1}+\cdots+ a_N x_{n-N}$, by setting $x_n \mathrel{{:}{=}} y_n-y^*$ if $y^* \mathrel{{:}{=}} b/(1-a_1-\cdots -a_N)$ exists. In particular, the sequence $y_n=y_{n-1}+\cdots +y_{n-N}+b$ can be homogenised, yielding the $N$\nobreakdash-bonacci recurrence relation $x_n=x_{n-1}+\cdots +x_{n-N}$; its closed\/-\/form solution is $x_n=c_1 \lambda_1^n + \cdots +c_N \lambda_N^n$, where $c_i\in \BBC$ are constants (that depend on the initial data) and each $\lambda_i\in \BBC$ is a root of the characteristic polynomial
\begin{equation}\label{eq:NbonPolynomial}
    p(\lambda) = \lambda^N - \lambda^{N-1} - \cdots - \lambda - 1.
\end{equation}
The task is to determine the roots of $p(\lambda)$. We claim 
that $N-1$ roots lie within the complex unit disk, whereas the $N$th root, denoted by $\alpha_N$, lies in the interval $(1,2)\subsetneq \BBR$, 
see Fig.~\ref{fig:rootlowerbounds} on p.~\pageref{fig:rootlowerbounds} below. 
If the coefficient $c_N \mathrel{{=}{:}} C$ preceding $\alpha_N$ (and determined by the initial datum) is non-zero, the sequence $\{x_n\}$ is, asymptotically, $x_n\simeq
C\alpha_N^n$, which diverges exponentially to $+\infty$ (if $C>0$) as $n\to+\infty$.

\begin{propNo}[\protect{\cite[Lemma~3.6]{Wolfram98}}: growth of $N$\nobreakdash-bonacci numbers]\label{prop:NbonGrowth}
    The $N$\nobreakdash-bonacci numbers~$F^{(N)}_n$, given by the recurrence 
$F^{(N)}_n = F^{(N)}_{n-1} + \cdots + F^{(N)}_{n-N}$ from     
$F^{(N)}_0=\cdots = F^{(N)}_{N-2}=0$ and $F^{(N)}_{N-1}=1$,
grow asymptotically as $F^{(N)}_n \simeq C\alpha_N^n$, where
Wolfram's lower bound is $2(1-2^{-N}) < \alpha_N < 2$ and $C>0$.
    (All other roots $\lambda_i$ of 
characteristic polynomial~\eqref{eq:NbonPolynomial} satisfy the bounds $3^{-1/N}<|\lambda_i|<1$, 
i.e.\ they lie within the annulus close to the boundary of the complex unit disk.)

\noindent$\bullet$\quad 
The numbers $\alpha_N$ are monotonically increasing w.r.t. $N$ as $\varphi = \alpha_2 < \alpha_3 < \cdots < \alpha_\infty = 2$. 
\end{propNo}

In Appendix~\ref{SecAppReProof} we re\/-\/prove this fact using Rouch\'e's theorem~\cite[\S\,10.43\:(b)]{RudinRC} from complex analysis, obtaining a slightly smaller than Wolfram's, but asymptotically equal (as~\mbox{$N\to \infty$}) lower bound for $\alpha_N = 2 (1-\veps_N)$: 
our bound is $\varepsilon_N < \delta_N \mathrel{{:}{=}} \tfrac{1}{2} (1+1/N)^N 2^{-N}$, which yields the limit $\lim_{N\to+\infty} (\delta_N/2^{-N})=e/2\approx 1.359$
that compares our bound with Wolfram's. Our approach is very elementary and geometric; it could be refined further to find a sharper 
bound.

\begin{rem}[Pisot\/--\/Vijayaraghavan number]\label{rem:PVnumber}
    The number $\alpha_N$ is called a Pisot\/--\/Vijayaraghavan (PV) number as it is the only real root $\alpha_N>1$ of the monic irreducible polynomial $p(\lambda)\in\BBZ[\lambda]$ whose all other roots lie within the complex unit disk.
It is known that as $n\to+\infty$, the powers $\alpha_N^n$ of the PV numbers are `almost integers' whose distance to the nearest integer tends to~$0$ as fast as~$\exp(-n)$. So, to get (a precise approximation of) the integer $F_n^{(N)} = C \alpha_N^n + (\exp\text{small terms})$, it suffices to know the constants $C$ and $\alpha_N$ and the index $n\in\BBN$. Conversely, for large $n$ (hence $F_n^{(N)} \gg 1$), one easily gets (a good approximation of $C$ and then) the index $n$ for a given $N$-bonacci number~$F_n^{(N)}$.\\[0.5pt]
\centerline{\rule{3.5in}{0.7pt}}
\end{rem}

\begin{rem}[upper bound satisfied]\label{rem:upperBoundSatisfied}
    Notice that in Examples~\ref{ex:FiboBase} and~\ref{ex:Nbon} -- corresponding to the Wronskians over $\BBR^1$ of odd differential order $k$ -- we have attained the upper bound for the growth in degree: the monomials' degrees \emph{are} the $N$\nobreakdash-bonacci numbers, and their coefficients never vanish. 
Simultaneously, it is an \textbf{open problem} whether the bound can be attained --\,with nonzero coefficient\,-- over higher base dimension $d=\dim \BBR^d > 1$ and any differential order $k\geqslant 1$. 
(In Example~\ref{ex:FiboOverDim} the coefficients are seen to be nonzero
both in~\eqref{EqD2CoeffFiboPlus} over $d=2$ and in~\eqref{eq:FiboOverDim} over $d\geqslant 3$, yet these were specific to the differential order $k=1$ in the Wronskian bracket~$W_d^k$.)
\end{rem}

\begin{ex}[computer algebra experiments]\label{ex:computerAlgebra}
  Using the computer algebra (\texttt{SymPy}) script in Appendix~C of~\cite{FinDim} --\,for calculating the complete generalised Wronskian determinants of monomial arguments,\,-- we 
made 32 experiments over base dimensions $d=1,2,3,4$ and differential orders $k=1,2,3$ by choosing an initial $N$-tuple of monomial arguments and iterating the bracket: re-inserting the newly-created monomials as arguments;
these computations are reported in the ancillary 
file
.\footnote{\label{FootFormatData}
The format of each entry in the output is this: the base dimension $d\geqslant 1$, the differential order $k\geqslant 1$, the arity $N=\binom{d+k}{d}$ of the complete generalised Wronskian; 
after the $N$-tuple of initially taken monomial ar\-gu\-ments, 
we describe the scheme: which of the arguments are fixed and how the newly produced monomials are re\/-\/used.
     In this setting, we write the starting 50 coefficients of the monomials which are produced by ite\-ra\-ting the Wronskian~$W_d^k$. We report the overall number of ite\-ra\-ti\-ons actually run in the example at hand and finally, the --\,sometimes, huge\,-- coefficient obtained in the last iteration.
(It is the threshold of circa 10\,000 iterations on or near which the coefficients get so big that the software hits the maximal operationally admissible num\-ber of digits in a decimal number. This is why the actual number of reported iterations is 
less than~10\,000.)}

Fix the base dimension $d$ of $\BBR^d\ni \bx=(x^1,\ldots,x^d)$ and differential order $k$ of the Wronskian determinant $W_d^k$. Take some $N=\tbinom{d+k}{d}$\nobreakdash-tuple $(m_1(\bx),\ldots,m_N(\bx))$ of monomial arguments $m_n(\bx)\in\BBR[x^1,\ldots,x^d]$ and denote by $b_n \mathrel{{:}{=}}\deg_{i}(m_n)$
their degrees in some coordinate~$x^i$; set up the scheme of iteration for the indices $n\geqslant N+1$ by $m_{n}(\bx) \mathrel{{:}{=}} W_d^k(m_{n-N},\ldots,m_{n-1})$. By~\cite[Theorem~7 on p.~12]{FinDim}, the monomials' degree $b_n$ in each base variable satisfies the linear recurrence relation $b_n=b_{n-1}+\cdots+b_{n-N}-kN/(d+1)$, which, upon homogenisation, becomes the recurrence relation for the $N$\nobreakdash-bonacci numbers. The monomials' coefficients satisfy their own (different, if any) recurrence relation.

    For instance, deploy 
the Wronskian $W_{d=2}^{k=1}=\Bone\wedge \partial/\partial x \wedge \partial/\partial y$ of first order over \mbox{$\BBR^2\ni (x,y)$} as the ternary ($N=3$) bracket; take as the three arguments the (monic) monomials $x^{a_n}y^{a_{n+1}}$, $x^{a_{n+1}}y^{a_{n+2}}$, and $x^{a_{n+2}}y^{a_{n+3}}$, where the integers $a_n$ are given by the ternary recurrence relation $a_n=a_{n-1}+a_{n-2}+a_{n-3}-1$ starting from a suitably chosen initial datum $(a_1,a_2,a_3)$. The Wronskian of these three monomials is the monomial $x^{a_{n+3}}y^{a_{n+4}}$ with coefficient $c_n$.
We verified using computer algebra that the coefficients $\{c_n\}_{n\in\BBN}$ satisfy the recurrence relation $c_{n}=-c_{n-1}-c_{n-2}+c_{n-3}$, whence the integers $c_n$ oscillate between large positive and negative values. It is unknown whether an initial datum $(a_1,a_2,a_3)$ exists such that the coefficients $c_n$ never vanish. 
We found many choices of initial data such that $c_n\neq 0$ up to $n\approx 3\,000$.
\end{ex}

\begin{rem}[Skolem\/--\/Pisot problem]\label{rem:Skolem}
   Given a linear recurrence relation for $z_n\in \BBZ$ and an initial datum, the Skolem\/--\/Pisot problem is to determine 
whether the sequence of integers $\{z_n\}_{n\in\BBN}$ ever hits zero: $z_n=0$. 
If zeros exist, it is known (the Skolem\/--\/Mahler\/--\/Lech Theorem, see~\cite{SkolemNP}
) that the set of indices $n\in \BBN$ for which $z_n=0$ consists of the union of a (possibly empty) finite set with a set of arithmetic progressions. In general, it is NP-hard to decide whether the sequence $\{z_n\}_{n\in\BBN}$ will ever hit a zero~\cite{SkolemNP}.
\end{rem}

\begin{conjectureNo}\label{conj:attainibility}
    Take the Wronskian $W_d^k$ of any differential order $k\geqslant 1$ over any base dimension $d=\dim \BBR^d\geqslant 1$ as the $N=\tbinom{d+k}{d}$\nobreakdash-ary bracket on the space of polynomials $\BBR[x^1,\ldots,x^d]$. 
We conjecture that there exists an $N$\nobreakdash-tuple of monomials $(m_1(\bx),\ldots,m_N(\bx))\in \BBR[x^1,\ldots,x^d]$ such that the scheme 
$m_{n}(\bx) \mathrel{{:}{=}} W_d^k(m_{n-N},\ldots,m_{n-1})$
to iterate 
the bracket over $n\in\BBN$ yields the monomials $m_n(\bx)$ whose degrees (total or in some, or in every variable) grow asymptotically as fast as the $N$\nobreakdash-bonacci numbers (by construction) and \emph{whose coefficients are never zero: $m_n(\bx)\not\equiv 0$ for all $n\in \BBN$.}
\end{conjectureNo}

\subsection*{Conclusion}
The asymptotic bound (see our theorem on p.~\pageref{thm:main}) of polynomial degrees that are obtained by iterating the Wronskian~$W_d^k$ as the $N=\binom{d+k}{d}$-ary bracket on $\BBR[x^1,\ldots,x^d]$ implies the asymptotic bound in Eq.~\eqref{EqDimBoundAsympt} for the dimensions of vector subspaces $\EuA_n\subsetneq \BBR[x^1$, $\ldots$, $x^d]$ spanned after the $n$th iteration.
One would expect that, whatever be the schemes to re\/-\/use the newly created monomials, far not all of the lower\/-\/degree monomials --\,within the degree bound\,-- will effectively be produced and available after the $n$th step (if ever at all); if so, the actual dimensions of subspaces $\EuA_n$ will tend to be smaller.
We pose the problem of improving --\,i.e.\ lowering\,-- the upper bound upon the growth rate for the subspace dimensions $\dim\EuA_n$ in the (strongly homotopy Lie) algebras of polynomials on~$\BBR^d$.

Our present choice to first study the polynomials --\,as objects' coefficients\,-- made the research problem clearly distinct from Kirillov's original formulation in~\cite{KirillovKontsevich} for generic vector fields with smooth coefficients on the line~$\BBR$.
Unlike the generic tuples of initially taken functions $f_1$,\ $\ldots$,\ $f_{r\geqslant N} \in C^\infty(\BBR^d)$ which, by definition, are not constrained by any differential\/-\/algebraic relations,\footnote{\label{FootJacobiNotConstrain}
Each Jacobi\/(-type) identity of the (strongly homotopy) Lie algebra constrains the structure, with respect to which the identities are quadratic (see~\cite{PRG25,Yer25} and many references therein), yet none of the Jacobi identities restricts the set of its arguments.
For example, see~\cite{ForKac,PRG25} for the Jacobi identities which either of the ternary brackets in~\eqref{EqExTernaryD1D2} satisfies for \emph{arbitrary} 5-tuples of smooth functions on the respective base space~$\BBR^d$.} 
the polynomials in~$\EuA_0$ have their degrees bounded 
uniformly by~${\mathsf{D}}_0^{(N)}$. (Naturally, all sufficiently high order derivatives of every polynomial function vanish identically, to begin with.)
The risk to encounter a slower rapidity of the degree growth (if it does, and apart from reproduction of the earlier reached monomials) is, in the case of $\BBR[x^1,\ldots,x^d]\subsetneq C^\infty(\BBR^d)$, due to the abrupt termination of those monomial sequences which hit the zero coefficient.
The problem thus doubles, now incorporating Kirillov's question about the dimensions $\dim\EuA_n$ and the Skolem\/--\/Pisot problem about the coefficients
(of the monomials that span~$\EuA_n$).
Our conjecture, see above, means that the growth does persist for generic 
tuples of initially taken 
polynomials on~$\BBR^d$.
It would be interesting to compare the observed growth rates for the case of polynomials and their $N$-ary brackets $W_{d\geqslant 1}^{k\geqslant 1}$ with the intermediate growth rate for the case of smooth functions as vector field coefficients and their binary Lie bracket~$W_{1}^{1}$ in~\cite{KirillovKontsevich}.

{\small\subsubsection*{Acknowledgements} 
The first author thanks the Center for Information Technology of the University of Groningen for access to the High Performance Computing cluster, H\'abr\'ok.
The second author thanks M.\,Kontsevich for suggesting to extend the research problem in~\cite{KirillovKontsevich} from binary Lie brackets of vector fields to the $N$-ary Lie\/-\/type brackets given on (the algebra of coefficients) $C^\infty(\BBR^d) \supsetneq \BBR[x^1,\ldots,x^d]$ by the generalised Wronskians $W_{d\geqslant 1}^{k\geqslant 1}$ of differential order~$k$, and for stimulating discussions at the~$\smash{\text{IH\'ES}}$.

}

\appendix
\section{Proof of the main Theorem}\label{SecAppProofThm}
\begin{proof}[Proof (of the main Theorem)]\label{pf:mainThm}
    It suffices, by multilinearity, to assume that 
the set of initial arguments consists of monomials. 
By~\cite[Theorem~7 on p.~12]{FinDim}, the Wronskian of monomial arguments $p_1(\bx),\ldots,p_N(\bx)$ is the monomial $p_1\cdots p_N/\prod_{i=1}^d (x^i)^{kN/(d+1)}$ preceded by the constant given by the generalised Vandermonde determinant of the monomials' degrees in each variable. We find an upper bound for the degree in each variable (and hence the total degree) by presently assuming that this constant never vanishes.

Iterate the Wronskian with the non\/-\/vanishing constant $n$ times and consider a monomial $q_n(\bx)$ with the highest total degree ${\mathsf{D}}_n^{(N)}=\deg q_n(\bx)$; it was produced by some sequence of re-used monomial arguments. If the newly\/-\/created argument was not used immediately 
in the next iteration, the growth was necessarily not faster 
than maximal. Let us consider which of the old arguments to drop in place of the new one; by the mechanism of degree growth, dropping an 
argument with the smallest degree results in the monomial with the highest degree. The degrees $a_n$ in each base variable of $q_n$ satisfy the recurrence $a_n=a_{n-1}+\cdots+a_{n-N}-kN/(d+1)$. By the properties of linear recurrences, after homogenisation, the closed form solution is $a_n=a^* + C\alpha_N^n + \sum_{j=1}^{N-1} c_j\lambda_j^n$, where $\alpha_N>1$ and $a^*,C\in \BBR$, $c_j\in \BBC$ are constants, and where all the values $\lambda_j$ lie in the complex unit disk: $|\lambda_j|<1$. (See the Proposition on p.~\pageref{prop:NbonGrowth} and the accompanying discussion%
.)

Hence 
${\mathsf{D}}_n^{(N)}=a_n^{(1)}+\cdots +a_n^{(d)}=(C^{(1)}+\cdots+C^{(d)})\alpha_N^n + d\cdot a^*+\sum_{i,j}c_j^{(i)}\lambda_j^n$, where $C^{(i)}$ is the constant $C$ for the recurrence of degrees $a_n$ in base variable $x^i$. Therefore, assuming for the uppper bound that an initial condition exists with $C^{(i)}\neq 0$ for all $i$, we have $\lim_{n\to+\infty} \bigl( {\mathsf{D}}_n^{(N)}/F_n^{(N)}\bigr)=C^{(1)}+\cdots+C^{(d)}<\infty$, completing the proof.
\end{proof}

\section{Re-proof of the Proposition and comparison of the bounds}\label{SecAppReProof}
\begin{proof}[Re-proof (of the Proposition).]\label{proof:Prop}
    The characteristic polynomial $p(z)$ of the $N$\nobreakdash-bonacci recurrence is $p(z)=z^N-z^{N-1}-\cdots-z-1$ (cf.\ Eq.~\eqref{eq:NbonPolynomial}).
Let us analyse where its roots are located in the complex plane.
(For example, the numerical values of the roots of $p(z)$ and their lower bounds are shown in Fig.~\ref{fig:rootlowerbounds} for $N=2$ (Fibonacci), $N=3$ (tribonacci), $N=6$, and $N=11$.)
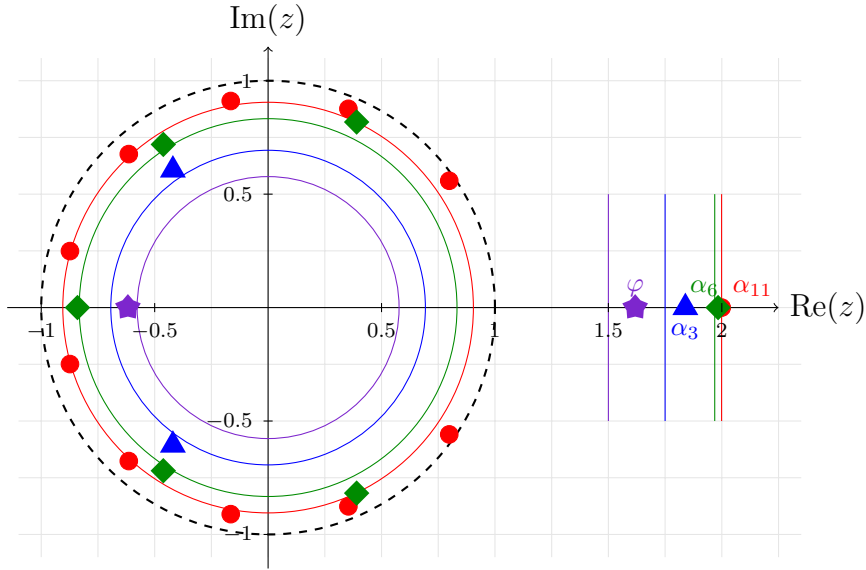
\begin{figure}[htb]
    \centering
    \begin{tikzpicture}[scale=3]
    \draw[step=0.25, gray!20, thin] (-1.1,-1.1) grid (2.35,1.1);
    \draw[->] (-1.15,0) -- (2.25,0) node[right] {$\mathrm{Re}(z)$};
    \draw[->] (0,-1.15) -- (0,1.15) node[above] {$\mathrm{Im}(z)$};

    \draw[thick, dashed] (0,0) circle (1);
    
    \draw[red] (0,0) circle (0.904951576);
    \draw[Green4] (0,0) circle (0.832683178);
    \draw[blue] (0,0) circle (0.693361274);
    \draw[Purple3] (0,0) circle (0.577350269);

    \draw[red] (1.999023438,-0.5) -- (1.999023438,0.5);
    \draw[Green4] (1.96875,-0.5) -- (1.96875,0.5);
    \draw[blue] (1.75,-0.5) -- (1.75,0.5);
    \draw[Purple3] (1.5,-0.5) -- (1.5,0.5);

    \foreach \x/\y/\lbl in {
        1.99951/0/$\alpha_{11}$
    }{
        \filldraw[red] (\x,\y) circle (1.1pt);
        \node[above right, red, font=\footnotesize] at (\x,\y) {\lbl};
    }
    \foreach \x/\y/\lbl in {
        -0.873392/0.249013/$z_1$,
        -0.614054/0.676538/$z_1$,
        -0.164746/0.910538/$z_1$,
        0.353942/0.876137/$z_1$,
        0.798495/0.559020/$z_1$,
        -0.873392/-0.249013/$z_1$,
        -0.614054/-0.676538/$z_1$,
        -0.164746/-0.910538/$z_1$,
        0.353942/-0.876137/$z_1$,
        0.798495/-0.559020/$z_1$
    }{
        \filldraw[red] (\x,\y) circle (1.1pt);
    }

    \foreach \x/\y/\lbl in {
        1.9835828434243263303856293/0/$\alpha_{6}$
    }{
        \node[regular polygon, regular polygon sides=4, fill=Green4, rotate=45, scale=0.6] at (\x,\y) {};
        \node[above, Green4, font=\footnotesize] at (\x-0.06,\y) {\lbl};
    }
    \foreach \x/\y/\lbl in {
        -0.840309/0/$z_1$,
        -0.461929/0.719144/$z_1$,
        -0.461929/-0.719144/$z_1$,
        0.390292/0.817862/$z_1$,
        0.390292/-0.817862/$z_1$
    }{
        \node[regular polygon, regular polygon sides=4, fill=Green4, rotate=45, scale=0.6] at (\x,\y) {};
    }

    \foreach \x/\y/\lbl in {
        1.83928675521416/0/$\alpha_{3}$
    }{
        \node[regular polygon, regular polygon sides=3, fill=blue, scale=0.5] at (\x,\y) {};
        \node[below, blue, font=\footnotesize] at (\x,\y-0.02) {\lbl};
    }
    \foreach \x/\y/\lbl in {
        -0.41964/0.60629/$z_1$,
        -0.41964/-0.60629/$z_1$
    }{
        \node[regular polygon, regular polygon sides=3, fill=blue, scale=0.5] at (\x,\y) {};
    }

    \foreach \x/\y/\lbl in {
        1.618034/0/$\varphi$
    }{
        \node[star, star points=5, fill=Purple3,scale=0.6] at (\x,\y) {};
        \node[above, Purple3, font=\footnotesize] at (\x,\y) {\lbl};
    }
    \foreach \x/\y/\lbl in {
        -0.618034/0/$z_1$
    }{
        \node[star, star points=5, fill=Purple3,scale=0.6] at (\x,\y) {};
    }

    \foreach \x in {-1,-0.5,0.5,1,1.5,2}
    \draw (\x,0.02) -- (\x,-0.02) node[below, font=\tiny] {$\x$};
    \foreach \y in {-1,-0.5,0.5,1}
    \draw (0.02,\y) -- (-0.02,\y) node[left, font=\tiny] {$\y$};
\end{tikzpicture}
    \vspace{-0.8cm}
    \caption{Real and complex roots of the characteristic polynomial for the $N$\nobreakdash-bonacci numbers $p(z)=z^N-z^{N-1}-\cdots-z-1$, plotted for the values $N=2$ (Fibonacci numbers, golden ratio \textcolor{Purple3}{$\varphi\approx 1.61$}, as \textcolor{Purple3}{\textbf{purple stars}}), $N=3$ (tribonacci numbers, \textcolor{blue}{$\alpha_3\approx 1.84$}, as \textcolor{blue}{\textbf{blue triangles}}), $N=6$ (\textcolor{Green4}{$\alpha_6\approx 1.984$}, as \textcolor{Green4}{\textbf{green rhombi}}), and $N=11$ (\textcolor{red}{$\alpha_{11}\approx 1.9995$}, as \textcolor{red}{\textbf{red circles}}). Shown as circles and straight lines are the lower bounds $\ell_N=3^{-1/N}$ and $L_N=2(1-2^{-N})$ for the complex modulus of the roots $\lambda_1$,\ $\ldots$,\ $\lambda_{N-1}$ in the complex unit disk and the positive real root $\alpha_N$, respectively; the lower bounds for the complex roots' moduli are \textcolor{Purple3}{$\ell_2\approx0.577$}, \textcolor{blue}{$\ell_3\approx0.693$}, \textcolor{Green4}{$\ell_6\approx0.832$}, and \textcolor{red}{$\ell_{11}\approx0.904$}, whereas for the positive real root the lower bounds are: \textcolor{Purple3}{$L_2=1.5$}, \textcolor{blue}{$L_3=1.75$}, \textcolor{Green4}{$L_6\approx1.968$}, and \textcolor{red}{$L_{11}\approx1.99902$}. }
    \label{fig:rootlowerbounds}
\end{figure}%
This reasoning consists of two parts.    
    
\smallskip
\textbf{1.}\quad   
To demonstrate that $N-1$ roots of $p(z)$ lie within the complex unit disk,
we recall the argument 
from~\cite{MathOverflow}.  
Rewrite the polynomial $P(z) \mathrel{{:}{=}} (1-z) p(z)=z^{N+1}-2z^N+1=(z^{N+1}+1)+(-2z^N)$ as the sum of two holomorphic functions on $D_{r>1}=\{z\in\BBC: |z|<r\}$. On the boundary $\partial D_r=\{z:|z|=r\}$ we have (for $r>1$ close to $1$), by the triangle inequality (which asserts the first inequality below),
    \begin{equation}\label{eq:rouche-inequality}
        |z^{N+1}+1|\leqslant r^{N+1} + 1 < 2r^N = |-2z^N| ;
    \end{equation}
as $P(1)=0$ and $P'(1)=1-N$ is negative, 
the polynomial $P(x)$ stays negative on some one\/-\/sided interval 
$\bigl( 1, r+{\Delta r}_{>0} \bigr)$
close to~$1$ that includes $r>1$, whence the second inequality above.
Now, as $\partial D_r$ is a simple curve in $\BBC$ and we have 
inequality~\eqref{eq:rouche-inequality}, it follows by Rouche's theorem~\cite[\S\,10.43\,(b)]{RudinRC} that $(z^{N+1}+1)+(-2z^N)$ and $(-2z^N)$ have the same number of zeros inside $D_r$, counted with multiplicities. Taking the limit $r\to 1$ from above yields $N$ zeros of $P(z)$ inside the closed unit disk $D_{\bar 1}=\{z:|z|\leqslant 1\}$, one of which is $z=1$ (indeed, $P(1)=1-2+1=0)$. Let us check that no other roots lie on the boundary of the unit disk: the equation $0=P(e^{i\theta}) = e^{iN\theta}(e^{i\theta}-2)+1$ yields, after rearranging and taking the norm, $|2-e^{i\theta}|=1$, which occurs only at $\theta=0$. Therefore, the polynomials $P(z)$ and $p(z)$ have $N-1$ roots $\lambda_1,\ldots,\lambda_{N-1}$ inside the open unit disk $D_1=\{z:|z|<1\}$, leaving one (real) root of $P(z)$ and $p(z)$ outside the closed unit disk.

\smallskip
\textbf{2.}\quad   
Now we bound 
this big real root. Restrict $P(x)$ to $x\in [1,2]\subsetneq \BBR$, where $P(1)=0$, $P(2)=1$, and $P'(x)=x^{N-1}\bigl((N+1)x-2N\bigr)$. As $P'(1)<0$, the function $P(x)$ is negative over some interval $(1,1+\varepsilon)$ close to 1. The extrema of $P(x)$ are at $x_1=0$ and $x_2=2N/(N+1)>1$. As there is only one extremum in $(1,2)$, it must be a minimum. By the continuity of the polynomial $P(x)$, there is 
a unique root $\alpha_N\in (x_2,2)\subsetneq (1,2)$ of $P(x)$ (and of $p(x)$). As the second derivative $P''(x)$ vanishes only at $x=0$ and $x=x_2-2/(n+1)<x_2$, it follows that $P(x)$ is convex up over $(x_2,+\infty)\supsetneq (x_2,2]$. Thus the root $\alpha_N$ must lie to the right of the intersection of the $x$-axis and 
straight line connecting the points $(x_2,P(x_2))$ and $(2,P(2))$ on the plane.
This line is given by the formula $y(x)=P(x_2)+ \bigl(1-P(x_2)\bigr) (x-x_2)/(2-x_2)$; the function $y(x)$ has the root at
\[
    x_3=x_2+ 
        {(2-x_2)} \big/ {\bigl(1-1/P(x_2)\bigr)} = 2-x_2^{-N}>x_2=2N/(N+1),
\]
where the second (short
) expression of~$x_3$ is obtained by simplification.
Hence $\alpha_N\in (x_3,2)\subsetneq (x_2,2)\subsetneq (1,2)\subsetneq \BBR$ and $\lim_{N\to+\infty}\alpha_N=2$. 

Writing the closed-form solution for the recurrence $F_n^{(N)}=C\alpha_N^n + \sum_{j=1}^{N-1}c_j\lambda_j^n$ with $C>0$ (and $c_j,\lambda_j\in\BBC$) completes the proof of the proposition.
\end{proof}

\enlargethispage{0.85\baselineskip}
\begin{rem}\label{RemCompareWithWolframs}
Write $\alpha_N=2(1-\varepsilon_N)$, where $\varepsilon_N\in (0,1-x_3/2)=(0,x_2^{-N}/2)$. Let us compare this upper bound for $\varepsilon_N$ with Wolfram's~\cite[Lemma\,3.6]{Wolfram98}: $\varepsilon_N\in (0,2^{-N})$. Expanding~$x_2^{-N}/2$~yields
$
    \delta_N \mathrel{{:}{=}} x_2^{-N}/2 = \tfrac{1}{2}\bigl( 1+ 
    {1}/{N}\bigr)^N 2^{-N} $,
hence $\lim_{N\to+\infty} (\delta_N/2^{-N}) = e/2 \approx 1.359$.
So, both the bounds are asymptotically proportional to each other and decay ex\-po\-nen\-ti\-al\-ly~fast. 
\end{rem}

\end{document}